\documentclass[11pt]{amsart} 
\usepackage{a4wide} 
\usepackage{amssymb} 
\usepackage{graphicx} 
\usepackage{MnSymbol}
\usepackage{wasysym} 

\usepackage{amsmath} 
\usepackage{amstext} 
\usepackage{amsfonts} 
\usepackage{amssymb} 
\usepackage{amsthm} 
\usepackage{amsrefs} 
\usepackage{mathrsfs} 
\usepackage{amsxtra} 
\usepackage{amscd} 
\usepackage{latexsym} 
\usepackage{verbatim} 
\usepackage{tikz} 
\usepackage{parskip} 
\usepackage{multicol}

\usepackage[colorlinks, 
             linkcolor=blue, 
             citecolor=black!75!red, 
             pdfproducer={pdfLaTeX}, 
             pdfpagemode=None, 
             bookmarksopen=true 
             bookmarksnumbered=true]{hyperref}

\usepackage{tikz}
\usetikzlibrary{arrows,calc,decorations.pathreplacing,decorations.markings,intersections,shapes.geometric,through,fit,shapes.symbols,positioning,decorations.pathmorphing}

\usepackage{amsmath,amsthm,stmaryrd} 
\usepackage{cleveref} 

\newtheorem{lemma}{Lemma}[section] 
\newtheorem{theorem}[lemma]{Theorem} 
\newtheorem{proposition}[lemma]{Proposition} 
\newtheorem{corollary}[lemma]{Corollary} 
\newtheorem{question}[lemma]{Question}

\theoremstyle{definition}  
 
\newtheorem{definitionnodiamond}[lemma]{Definition} 
\newtheorem{examplenodiamond}[lemma]{Example} 
\newtheorem{remarknodiamond}[lemma]{Remark}

\newenvironment{definition}{\begin{definitionnodiamond}}{\hfill\ensuremath\blacklozenge\end{definitionnodiamond}} 
\newenvironment{example}{\begin{examplenodiamond}}{\hfill\ensuremath\blacklozenge\end{examplenodiamond}} 
\newenvironment{remark}{\begin{remarknodiamond}}{\hfill\ensuremath\blacklozenge\end{remarknodiamond}}

\makeatletter 
\let\xx@thm\@thm 
\AtBeginDocument{\let\@thm\xx@thm} 
\makeatother

\renewcommand\qedhere{\qed}

\numberwithin{equation}{section}

\crefname{section}{Section}{Sections} 
\crefformat{section}{#2Section~#1#3}  
\Crefformat{section}{#2Section~#1#3}  
 
\crefname{subsection}{}{Subsections} 
\crefformat{subsection}{\S#2#1#3}  
\Crefformat{subsection}{\S#2#1#3}

\crefname{definition}{Definition}{Definitions} 
\crefformat{definition}{#2Definition~#1#3}  
\Crefformat{definition}{#2Definition~#1#3}  
 
\crefname{definitionnodiamond}{Definition}{Definitions} 
\crefformat{definitionnodiamond}{#2Definition~#1#3}  
\Crefformat{definitionnodiamond}{#2Definition~#1#3}

\crefname{example}{Example}{Examples} 
\crefformat{example}{#2Example~#1#3}  
\Crefformat{example}{#2Example~#1#3}  

\crefname{question}{Question}{Questions} 
\crefformat{question}{#2Question~#1#3}  
\Crefformat{question}{#2Question~#1#3}

\crefname{examplenodiamond}{Example}{Examples} 
\crefformat{examplenodiamond}{#2Example~#1#3}  
\Crefformat{examplenodiamond}{#2Example~#1#3}  
 
\crefname{remark}{Remark}{Remarks} 
\crefformat{remark}{#2Remark~#1#3}  
\Crefformat{remark}{#2Remark~#1#3}  
 
\crefname{remarknodiamond}{Remark}{Remarks} 
\crefformat{remarknodiamond}{#2Remark~#1#3}  
\Crefformat{remarknodiamond}{#2Remark~#1#3}  
 
\crefname{convention}{Convention}{Conventions} 
\crefformat{convention}{#2Convention~#1#3}  
\Crefformat{convention}{#2Convention~#1#3}  
 
\crefname{lemma}{Lemma}{Lemmas} 
\crefformat{lemma}{#2Lemma~#1#3}  
\Crefformat{lemma}{#2Lemma~#1#3}  
 
\crefname{proposition}{Proposition}{Propositions} 
\crefformat{proposition}{#2Proposition~#1#3}  
\Crefformat{proposition}{#2Proposition~#1#3}  
 
\crefname{corollary}{Corollary}{Corollaries} 
\crefformat{corollary}{#2Corollary~#1#3}  
\Crefformat{corollary}{#2Corollary~#1#3}  
 
\crefname{theorem}{Theorem}{Theorems} 
\crefformat{theorem}{#2Theorem~#1#3}  
\Crefformat{theorem}{#2Theorem~#1#3}  
 
\crefname{assumption}{Assumption}{Assumptions} 
\crefformat{assumption}{#2Assumption~#1#3}  
\Crefformat{assumption}{#2Assumption~#1#3}  
 
\crefname{equation}{}{} 
\crefformat{equation}{(#2#1#3)}  
\Crefformat{equation}{(#2#1#3)} 
 
\crefname{align}{}{} 
\crefformat{align}{(#2#1#3)}  
\Crefformat{align}{(#2#1#3)} 
 
\crefname{proofstep}{Step}{Steps} 
\crefformat{proofstep}{#2Step~#1#3}  
\Crefformat{proofstep}{#2Step~#1#3}

\DeclareMathOperator{\ind}{Ind}
\DeclareMathOperator{\res}{Res} 
\newcommand{\cst}{\ifmmode\mathrm{C}^*\else{$\mathrm{C}^*$}\fi} 
\newcommand{\wst}{\ifmmode\mathrm{C}^*\else{$\mathrm{W}^*$}\fi}

\newcommand\bC{\mathbb C}

\newcommand\bG{\mathbb G} 
\newcommand\bH{\mathbb H}

\newcommand\bS{\mathbb S}

\newcommand\bX{\mathbb X} 
 
\newcommand\bZ{\mathbb Z} 
 
\newcommand\cA{\mathcal A} 
\newcommand\cB{\mathcal B}

\newcommand\cE{\mathcal E}

\newcommand\cH{\mathcal H}

\newcommand\cL{\mathcal L}

\newcommand{\Rep}{\mathrm{Rep}}

\numberwithin{equation}{section}

\title{Ergodic actions of the compact quantum group $O_{-1}(2)$} 
\author{Alexandru Chirvasitu,\quad Souleiman Omar Hoche}

\begin{document} 
\date{}

\newcommand{\Addresses}{{
  \bigskip 
  \footnotesize 
		 
  \textsc{Department of Mathematics, University at Buffalo, Buffalo,
    NY 14260-2900, USA}\par\nopagebreak \textit{E-mail address}:
  \texttt{achirvas@buffalo.edu}

  \medskip 
 
  \textsc{Laboratoire de Math\'ematiques de Besan\c{c}on, Universit\'e de Bourgogne Franche-Comt\'e, 16, Route de Gray, 25030 Besan\c{c}on Cedex, France}\par\nopagebreak 
  \textit{E-mail address}: \texttt{hoche.souleiman\_omar@univ-fcomte.fr}

}} 
 
\maketitle

\begin{abstract}
  Among the ergodic actions of a compact quantum group $\mathbb{G}$ on
  possibly non-commutative spaces, those that are {\it embeddable} are
  the natural analogues of actions of a compact group on its
  homogeneous spaces. These can be realized as {\it coideal
    subalgebras} of the function algebra $\mathcal{O}(\mathbb{G})$
  attached to the compact quantum group.

  We classify the embeddable ergodic actions of the compact quantum
  group $O_{-1}(2)$, basing our analysis on the bijective
  correspondence between all ergodic actions of the classical group
  $O(2)$ and those of its quantum twist resulting from the monoidal
  equivalence between their respective tensor categories of unitary
  representations.
  
  In the last section we give counterexamples showing that in general
  we cannot expect a bijective correspondence between embeddable
  ergodic actions of two monoidally equivalent compact quantum groups.
\end{abstract} 
 

\vspace{.5cm}


\tableofcontents

\section*{Introduction}

Ergodic actions of compact groups on possibly noncommutative operator
algebras offer a natural bridge between dynamical systems and
non-commutative geometry. The topic has been studied extensively and
we could not do justice to the literature, but we mention here the
papers \cite{hls-erg,w1,w2,w3}, some of whose material will feature
below in various ways.

With the advent of compact quantum groups introduced and studied by
Woronowicz in \cite{wor-psd,wor-su2,wor-sun} the scope of topics
pertinent to the study of classical compact groups has expanded to
include these. In this context we mention \cite{bc-erg,wng-erg}, where
the authors study ergodic coactions
\begin{equation}\label{eq:intro-act}
  N\to N\otimes A
\end{equation}
of non-commutative ``function algebras'' $A$ of compact quantum groups
on (typically again non-commutative) operator algebras $N$, be it
$C^*$ or von Neumann.

Purely quantum phenomena arise: in stark contrast to ordinary compact
groups, compact quantum groups can act ergodically on type-III factors
(\cite[Corollary 3.7]{wng-erg}). Moreover, closer in spirit to the
contents of this paper, it is explained in \cite[$\S$6]{wng-erg} that
the underlying non-commutative spaces on which a compact quantum group
acts ergodically need not be a quotient by a quantum subgroup. 

The so-called {\it embeddable} ergodic actions constitute a class that
is intermediate between fully general and quotients by quantum
subgroups. In the language of coactions \Cref{eq:intro-act}
embeddability simply means that there is an embedding $N\to A$ that
respects the right $A$-coactions on both sides (see \Cref{se.prelim}
below for precise definitions).

In the present paper we study the class of embeddable ergodic actions
for the compact quantum group $O_{-1}(2)$ obtained by
``cocycle-twisting'' the usual orthogonal group $O(2)$ and fitting
into the family of deformed orthogonal groups $O_q(2)$ for
$q\in [-1,1]$, classifying such actions in \Cref{cor.cls}.

Using the theory of idempotent states (analogous to idempotent
measures on classical locally compact groups) and its relation to
embeddable ergodic actions (\cite{FS}), the authors of \cite{FST} show
that for the less-problematic values $-1<q\le 1$ the embeddable
ergodic actions of the $q$-deformations $U_{q}(2)$, $SU_{q}(2)$, and
$SO_{q}(3)$ do in fact all arise as quotients by quantum
subgroups. \Cref{cor.cls} shows that this contrasts markedly with the
situation for $O_{-1}(2)$.

Cocycle deformation does not alter the monoidal category of
representations of the compact quantum group \cite{BBC}, and
implements an equivalence between the categories of ergodic actions
\cite{de2010actions}.

The natural question arises of whether we also have a natural
bijective correspondence between embeddable ergodic actions of two
monoidally equivalent compact quantum groups.  We will prove below in
\Cref{se.cntrex} that the answer is negative in general, in the strong
sense that even for finite monoidally equivalent quantum groups with
equidimensional underlying function algebras the numbers of
isomorphism classes of embeddable ergodic actions need not be equal.

The paper is structured as follows.

\Cref{se.prelim} contains preparatory material to be used throughout
the paper.

In \Cref{se.$O_{-1}(2)$} we study the ergodic actions of the cocycle
twist $O_{-1}(2)$ and classify those that are embeddable
(\Cref{cor.cls}). We then also describe them in terms of quantum
subgroups of $O_{-1}(2)$ and generalizations thereof (see
\Cref{subse.gen-quot}).
 
Finally, \Cref{se.cntrex} is concerned with studying to what extent
embeddable ergodic actions transport over to a cocycle twist. We will
see in \Cref{cor.nogo2} that even for finite quantum groups, this can
fail in a very strong sense. Along the way, we analyze the cocycle
twists of the dihedral groups $D_K$ analogously to $O_{-1}(2)$ and as
discrete versions of the latter.

\subsection*{Acknowledgements}

We thank Uwe Franz for his insights and advice on the contents of this
paper.

This project was inspired by an interesting talk on
\emph{actions of compact quantum groups\cite{commer:2016}} given by
Kenny De Commer at Bedlewo ( 28 June - 11 July, 2015).  We are
grateful to him both for his talk and for a very useful subsequent
discussion during the workshop "Quantum groups from combinatorics to
analysis" (Caen, 2016).

A.C. is grateful for funding through NSF
grant DMS-1565226.

S.O.H acknowledge support by MAEDI/MENESR and DAAD through the PROCOPE
program.

\section{Preliminaries}\label{se.prelim}

We will need some background on coalgebras and Hopf algebras; for
this, we refer the reader to any of the numerous good sources on the
subject: e.g. \cite{swe,mont,rad}.

Our algebras are all unital, and unless specified otherwise the
`$\otimes$' symbol denotes minimal tensor products when placed between
operator algebras ($C^*$ or, in rare cases, von Neumann algebras) and
the plain, algebraic tensor product when placed between
non-topological algebras.

Given functionals $\phi_i$, $i=1,2$ on a coalgebra $C$ with
comultiplication $\Delta$ we denote by $\phi_1*\phi_2$ their {\it
  convolution} defined by
\begin{equation*}
  \begin{tikzpicture}[auto,baseline=(current  bounding  box.center)]
    \path[anchor=base] (0,0) node (h) {$C$} +(2,.5) node (hh) {$C\otimes C$} +(4,0) node (c) {$\bC$};
    \draw[->] (h) to[bend left=6] node[pos=.5,auto] {$\scriptstyle \Delta$} (hh);
    \draw[->] (hh) to[bend left=6] node[pos=.5,auto] {$\scriptstyle \phi_1\otimes\phi_2$} (c);
    \draw[->] (h) to[bend right=6] node[pos=.5,auto,swap] {$\scriptstyle \phi_1*\phi_2$} (c);    
  \end{tikzpicture}
\end{equation*}

\subsection{Compact quantum groups}
\label{subse.cqg}

We adopt the notion of compact quantum group introduced by
Woronowicz. The present recollection will be very brief, as the theory
is quite expansive. We refer the reader to the excellent surveys
\cite{woronowicz,kt1} for background on the topic.

\begin{definition}\label{def.cqg-cast}
  A \textit{compact quantum group} is a pair $(A,\Delta)$ where $A$ is
  a unital $C^{*}$-algebra and $\Delta:A\rightarrow A\otimes A$ is a
  unital $*$-homomorphism which is
  coassociative:\[(\Delta\otimes id_{A})\circ \Delta=(id_{A}\otimes
    \Delta)\circ \Delta\] and $A$ satisfies the quantum cancellation
  properties:
  \[\overline{Lin}\left((1\otimes
      A)\circ\Delta(A)\right)=\overline{Lin}\left((A\otimes
      1)\circ\Delta(A)\right)=A\otimes A\]

  We denote by $A^{*}$ the set of states of $A$.  One of the most
  important features of compact quantum groups is the existence of a
  unique Haar state $h$, i.e a unique state on the $C^*$-algebra $A$
  such that
\begin{equation}
  \label{eq:haar}
  (h\otimes id_{A})\circ \Delta(a)=(id_{A}\otimes h)\circ
  \Delta(a)=h(a)1, \qquad  \forall a\in A.
\end{equation}

A compact quantum group is said of \emph{Kac type} if $h$ is tracial
i.e $h(ab)=h(ba),\ \forall a,b\in A$.
\end{definition}

The $C^*$-algebra $A$ underlying a compact quantum group has a unique
dense Hopf $*$-algebra $\cA$ (see \cite[Theorem 3.2.2]{kt1}), and much
of the theory of compact quantum groups can be phrased purely
algebraically, in terms of their underlying Hopf algebras
$\cA$. Abstractly, these objects were introduced in \cite{dk} and
following that source we use the following terminology to refer to
them.

\begin{definition}\label{def.cqg-alg}
  A {\it CQG algebra} is a complex Hopf $*$-algebra $\cA$ with a state
  $h:\cA\to \bC$ satisfying \Cref{eq:haar} and which is positive in
  the sense that $h(a^*a)\ge 0$ for all $a\in \cA$, with equality only
  at $a=0$.
\end{definition}

We can largely go back and forth between the $C^*$ and purely
algebraic context for studying compact quantum groups (see e.g. the
discussion in \cite[Sections 4 and 5]{dk}):
\begin{itemize}
\item On the one hand, as mentioned above, for any $C^*$-algebraic
  compact quantum group as in \Cref{def.cqg-cast} one can find a
  unique dense Hopf $*$-subalgebra that meets the criteria of
  \Cref{def.cqg-alg}.
\item Conversely, a CQG algebra has a universal $C^*$-completion that
  turns out to satisfy the requirements of \Cref{def.cqg-cast}.
\end{itemize}

\begin{remark}
  The $C^*$ envelope $C(\bG)$ from the above discussion is sometimes
 denoted by $C^u(\bG)$ (for {\it universal}), to distinguish it
  form other completions of $\cA(\bG)$ which in general exist and are
  also compact quantum groups in the sense of \Cref{def.cqg-cast}. We
  focus mainly on the universal setting, as sketched above.
\end{remark}

For the purposes of this paper it will be convenient to phrase things
primarily in terms of Hopf algebras, reverting to their $C^*$
envelopes whenever necessary. We denote compact quantum groups by bold
face letters such as $\bG$, by $C(\bG)$ the underlying $C^*$-algebra
of the compact quantum group and by $\cA(\bG)$ its dense CQG
algebra.

Moreover, we can also define a canonical von Neumann algebraic version
of $\bG$:

\begin{definition}\label{def.vn}
  $L^\infty(\bG)$ is the von Neumann closure of the GNS representation
  of $\cA(\bG)$ with respect to the Haar state $h$.
\end{definition}
Note that $L^\infty(\bG)$ comes equipped with a coassociative
comultiplication arising as the closure of
\begin{equation*}
  \Delta:\cA(\bG)\to \cA(\bG)\otimes \cA(\bG). 
\end{equation*}
When referring to generic compact quantum groups we will sometimes be
vague on which context we are in, unless it makes a difference.

Examples abound in the sources mentioned thus far; in this paper, the
main compact quantum group is the following ``twisted'' version of the
orthogonal group $O(2)$.

\begin{definition}\label{def-O2} 
  The compact quantum group $O_{-1}(2)$ is defined as the compact
  quantum group with underlying CQG algebra with self-adjoint
  generators $y=(y_{jk})_{1\le j,k\le 2}$ and the relations
\begin{enumerate} 
\item $y$ is orthogonal, i.e.\ the generators $y_{jk}$ are
  self-adjoint and satisfy the unitarity relations
  $y_{1j}y_{1k}+y_{2j}y_{2k}=\delta_{jk}=y_{j1}y_{k1}+y_{j2}y_{k2}$
  for $j,k=1,2$;
\item 
$y_{jk}y_{j\ell} = - y_{j\ell}y_{jk}$ and $y_{kj}y_{\ell j} = - y_{\ell j}y_{kj}$ for $k\not=\ell$; 
\item $y_{jk} y_{\ell m} = y_{\ell m} y_{jk}$ for $j\not= \ell$ and
  $k\not=m$.
\end{enumerate} 
The coproduct, counit and antipode of $O_{-1}(2)$ are given by
\begin{equation*}
  \Delta(y_{jk})=\sum_{i}y_{ji}\otimes y_{ik},\qquad
  \varepsilon(y_{jk})=\delta_{jk},\qquad S(y_{jk})=y_{kj}.  
\end{equation*}
\end{definition}

The notion of a quantum subgroup was introduced by Podle\'s
\cite{podles} for matrix pseudo-groups.

\begin{definition}
  Let $(A,\Delta_A)$ and $(B,\Delta_B)$ be two compact quantum
  groups. Then $(B,\Delta_B)$ is called a \emph{quantum subgroup} of
  $(A,\Delta_A)$, if there is exists a surjective $∗$-algebra
  homomorphism $\pi: A\rightarrow B$ such that
  $\Delta_B\circ \pi=(\pi\otimes \pi)\circ \Delta_A$.
\end{definition}


\begin{definition}
  A \emph{right coideal subalgebra} or {\it coidalgebra} $C$ in a
  compact quantum group $(A,\Delta)$ is a unital $∗$-subalgebra
  $C\subset A$ (complete when $A$ is a $C^*$-algebra) such that
  $\Delta(C)\subset C\otimes A$.
\end{definition}


\subsection{Actions}
\label{subse.act}

We aggregate here material on actions of compact quantum groups on
compact non-commutative spaces. The reader may consult
\cite{commer:2016} for a good survey of the field.

In general, we will denote by $\bG$ a compact quantum group realized
either as a $C^*$-algebra $C(\bG)$ as in \Cref{def.cqg-cast} or as a
CQG algebra $\cA(\bG)$ as described in \Cref{def.cqg-alg}.

Similarly, we denote by $\bX$ a {\it compact quantum} (or {\it
  non-commutative}) {\it space}, i.e. the object dual to a unital
$C^*$-algebra $C(\bX)$.

\begin{definition}
  Let $\bX$ and $\bG$ be as above.  A \textit{right action}
  $\mathbb{X}\overset{\alpha}{\curvearrowleft} \mathbb{G}$ is a
  $*$-morphism
  \begin{equation*}
    \alpha: C(\mathbb{X})\rightarrow C(\mathbb{X})\otimes C(\mathbb{G})
  \end{equation*}
  (a {\it coaction} of $C(\bG)$ on $C(\bX)$) which
  \begin{itemize}
  \item is coassociative in the sense that
    \[(\alpha\otimes Id_{\mathbb{G}})\circ
      \alpha=(id_{\mathbb{X}}\otimes \Delta)\circ \alpha,\] and
  \item
    $\alpha\left(C(\mathbb{X})\right)\left(\bC\otimes
      C(\mathbb{G})\right)$ is dense in
    $C(\mathbb{X})\otimes C(\mathbb{G})$.
  \end{itemize}   
\end{definition}

Keeping with the spirit of translating $C^*$-algebraic concepts into
purely algebraic ones, we note that given an action $\alpha$ as above
there is a dense $*$-subalgebra $\cA=\cA(\alpha)$ (or more improperly
$\cA(\bX)$, since it depends on $\alpha$ and not just $\bX$) such that
$\alpha$ is a completion of a comodule algebra structure denoted by
the same symbol:
\begin{equation}\label{eq:alg-act}
  \alpha: \cA\rightarrow \cA\otimes \cA(\mathbb{G}). 
\end{equation}

\begin{definition}\label{Cond} 
  Let $\mathbb{X}\overset{\alpha}{\curvearrowleft} \mathbb{G}$ and the
  \emph{quantum orbit space} $\mathbb{X}/\mathbb{G}$ i.e the
  $C^*$-algebra
  \begin{equation}\label{eq:orb}
    C(\mathbb{X}/\mathbb{G})=\{a\in C(\mathbb{X})\ |\ \alpha(a)=a\otimes 1\}.
  \end{equation}
\end{definition} 

Passing to the dense subalgebra $\cA=\cA(\alpha)\subseteq C(\bX)$
discussed above, it can be shown that the algebraic version of
\Cref{eq:orb} defined by
\begin{equation*}
  \cA(\bX/\bG):= \{a\in \cA\ |\ \alpha(a)=a\otimes 1\}
\end{equation*}
is dense in $\cA$.

\begin{definition}
  An action $\mathbb{X}\overset{\alpha}{\curvearrowleft} \mathbb{G}$
  is called \textit{ergodic} if $C(\mathbb{X}/\mathbb{G})=\mathbb{C}1$
  or equivalently, if $\cA(\bX/\bG)=\bC$.
\end{definition}

Let \Cref{eq:alg-act} be an ergodic algebraic action. It then turns
out that there is a unique state $h_\alpha$ on $\cA$ that is preserved
by the coaction. We can then form the von Neumann closure
$L^\infty=L^\infty(\cA,h_\alpha)$ of the GNS representation of $\cA$, and
\Cref{eq:alg-act} lifts to a coassociative von Neumann algebra
morphism
\begin{equation*}
  L^\infty\to L^\infty\otimes L^\infty(\bG)
\end{equation*}
(see \Cref{def.vn}).  Once again, we transition freely between the von
Neumann algebraic and the purely algebraic setting for ergodic
actions. The former features mostly in the classical context in
\Cref{subse.erg-clss} below, in order for us to connect with the
literature on ergodic actions of compact (plain, non-quantum) groups.

\begin{definition} 
  Let $\mathbb{X}\overset{\alpha}{\curvearrowleft} \mathbb{G}$. One
  calls $\alpha$ of \textit{quotient type} if there exists a compact
  quantum subgroup $\mathbb{H}\subset \mathbb{G}$ with corresponding
  quotient map $\pi: C(\mathbb{G})\rightarrow C(\mathbb{H})$ and a
  $*$-isomorphism
  \begin{equation*}
    \theta: C(\mathbb{X})\rightarrow C(\mathbb{H}\backslash\mathbb{G})=\left\{g\in
      C(\mathbb{G})\ |\ \left(\pi\otimes
        id_{\mathbb{G}}\right)\Delta(a)=1_{\mathbb{H}}\otimes
      a\right\}
  \end{equation*}
  such that
  \begin{equation*}
    (\theta\otimes id_{\mathbb{G}})\circ\alpha=\Delta\circ \theta,
  \end{equation*}
  i.e. such that $\theta$ respects the $C(\bG)$ coactions on the
  domain and codomain.
\end{definition}

Note that actions of quotient type are automatically ergodic. The
following definition captures a somewhat broader class of ergodic
actions.

\begin{definition} 
  An action $\mathbb{X}\overset{\alpha}{\curvearrowleft} \mathbb{G}$
  is \textit{embeddable} if there exists a faithful
  coaction-preserving $*$-morphism
  \begin{equation*}
    \theta: C(\mathbb{X})\hookrightarrow C(\mathbb{G})
  \end{equation*}
\end{definition} 
\begin{remark} 
  In other words, embeddable ergodic actions can be realized as
  coidalgebras in the Hopf algebra attached to the quantum group.
\end{remark}



\subsection{Monoidal equivalence}
\label{subse.mon}

The following notion of monoidal equivalence was introduced in
\cite{BRV} (see also \cite{de2010actions}).

\begin{definition}\label{def.mon}
  Two compact quantum groups $\bG_{1}=(A_{1},\Delta_{1})$ and
  $\bG_{2}=(A_{2},\Delta_{2})$ are said to be \textit{monoidally
    equivalent} if there exists a bijection
  $\psi: Irred(\bG_{1})\rightarrow Irred(\bG_{1})$ satisfying
  $\psi(\varepsilon)=\varepsilon$, together with linear isomorphisms
  \[\psi: Mor(x_1\otimes \cdots \otimes x_r, y_1\otimes \cdots \otimes
    y_k)\rightarrow Mor(\psi(x_1)\otimes \cdots \otimes \psi(x_r),
    \psi(y_1)\otimes \cdots \otimes \psi(y_k))\] satisfying the
  following conditions:

\[\psi(1)=1,\qquad \psi(S^{*})=(\psi(S))^{*},\qquad\psi(ST)=\psi(S)\psi(T),\qquad\psi(S\otimes T)=\psi(S)\otimes \psi(T)\]
for all $S\subset A_1\,\mbox{and}\,\, T\subset A_2$, whenever the
formulas make sense. Such a collection of maps $\psi$ is called a
\textit{monoidal equivalence} between $\bG_{1}$ and $\bG_{2}$.
\end{definition}

\begin{remark}
  The categories $\Rep(\bG_i)$ of unitary representations of $\bG_i$
  are monoidal $*$-categories, in the sense that there are complex
  conjugate-linear operators
  \begin{equation*}
      \begin{tikzpicture}[auto,baseline=(current  bounding  box.center)]
    \path[anchor=base] (0,0) node (l) {$\mathrm{hom}(x,y)$} +(3,0) node (r) {$\mathrm{hom}(y,x)$};
    \draw[->] (l) to node[pos=.5,auto] {$\scriptstyle *$} (r);
  \end{tikzpicture}
  \end{equation*}
  satisfying the obvious analogues of $*$ structures on
  $*$-algebras. Keeping this in mind, \Cref{def.mon} simply says that
  $\Rep(\bG_i)$ are equivalent as monoidal $*$-categories.
\end{remark}

\begin{remark}
  Concrete examples of monoidally equivalent compact quantum groups
  are given in section 4 of \cite{de2010actions}. As we will recall
  momentarily, our compact quantum group of interest $O_{-1}(2)$ is
  monoidally equivalent to $O(2)$.
\end{remark}

One source of monoidal equivalence is {\it cocycle
  twisting}. \cite{bch-coc} is an excellent source for the material
that we once more only skim here. As announced above, we work mainly
with plain, non-topologized Hopf algebras.

A {\it 2-cocycle} on a CQG algebra $\cH$ is map
$\lambda:\cH\otimes \cH\to \bC$ with convolution inverse
$\lambda^{-1}$ and satisfying certain associativity-like conditions
that specialize to it being a cocycle in the usual sense when
$\cH=\bC\Gamma$ is the group algebra of a discrete group (see
\cite[Example 1.3]{bch-coc}).

A $2$-cocycle allows us to deform the multiplication of $\cH$. In Sweedler notation
\begin{equation*}
  \cH\ni a\mapsto a_1\otimes a_2 = \Delta(a)\in \cH\otimes \cH
\end{equation*}
the deformed multiplication is
\begin{equation*}
  a\bullet b = \lambda(a_1,b_1)a_2b_2\lambda^{-1}(a_3,b_3).
\end{equation*}
The cocycle conditions ensure that this equips the underlying space of
$\cH$ with an associative algebra structure, and preserving the
comultiplication we obtain another CQG algebra $\cH^\lambda$ (see
\cite[$\S$3.3]{bch-coc}).

As explained in \cite[$\S$3.3]{bch-coc}, there is a monoidal
equivalence $\lambda\triangleright$ between the category of
$\cH$-comodules (i.e. $\Rep(\bG)$ if $\cH$ is the CQG algebra of the
compact quantum group $\bG$) and that of $\cH^\lambda$-comodules.

The instance of cocycle twisting that we are most concerned with here
is

\begin{example}
  Let $\cH$ be the CQG algebra of the orthogonal group $O(2)$, which
  surjects onto the CQG algebra $\bC \bZ_2^2$ of the diagonal
  subgroup of $O(2)$.

  Now, the $2$-cohomology $H^2(\bZ_2^2,\bC)$ is isomorphic to $\bZ/2$,
  and hence we can choose a $2$-cocycle that represents the unique
  non-trivial class. Such a cocycle then precomposes with the
  surjection
  \begin{equation*}
    \cH^{\otimes 2} \to (\bC\bZ_2^2)^{\otimes 2}
  \end{equation*}
  to give a 2-cocycle on $\cH$ in the sense of the present
  subsection. The twist $\cH^\lambda$ will be precisely the CQG
  algebra of $O_{-1}(2)$, as described in \Cref{def-O2}.
\end{example}

Recall the following paraphrase of \cite[Theorem 7.3]{de2010actions}.

\begin{theorem}\label{ergo_corres}
  A monoidal equivalence between the categories of representations of
  two compact quantum groups $\bG_i$, $i=1,2$ induces an equivalence
  between their categories of ergodic actions.
\end{theorem}

We denote by $\lambda\triangleright$ all such equivalences arising in
this context: the monoidal equivalence between categories of
representations, the equivalence between the categories of ergodic
actions, etc. Context will suffice to determine the correct
interpretation of the symbol $\lambda\triangleright$ in each case.
 
\Cref{ergo_corres} motivates
 
\begin{question}\label{qu.mon-erg} 
  Let $\bG_1$ and $\bG_2$ be two monoidally equivalent compact quantum
  group. Is there a natural bijective correspondence between their
  embeddable ergodic actions?
\end{question}

Even though the question is rather ill-posed and ambiguous, we will
see below that the answer is negative in as strong a sense as
possible, even for finite quantum groups.

\subsection{Ergodic actions of classical compact groups}
\label{subse.erg-clss}

Here we recall various generalities on ergodic actions of (ordinary,
non-quantum) compact groups on possibly non-commutative operator
algebras for later use. Our main references for all of this are the
seminal papers \cite{w1,w2,w3}.

We work in the context of actions on von Neumann algebras of compact
groups $\bG$ on von Neumann algebras (as in the papers referenced
above). In that setting, an action is {\it ergodic} if the fixed-point
subalgebra consists of scalars only.

The general theory of ergodic actions of compact groups on von Neumann
algebras is developed in \cite{w1} and deployed later in \cite{w2,w3}
for classification purposes. First, we recall the following simple
procedure for producing ergodic actions.

\begin{definition}\label{def.ind}
  Let $\bH\le \bG$ be an inclusion of compact groups and
  $\bH \overset{\alpha}{\curvearrowright} N$ an $\bH$-action on a von
  Neumann algebra. The {\it induced representation}
  $\ind_{\bH}^{\bG}(N)$ is the von Neumann algebra
  \begin{equation*}
    \{L^\infty(\bG)\otimes N \ni f:\bG\to N\ |\ f(gh^{-1})=\alpha_h(f(g))\}
  \end{equation*}
  equipped with the $\bG$-action given by
  \begin{equation*}
    g\triangleright f = f(g^{-1}\bullet)
  \end{equation*}
\end{definition}
Induction is the right adjoint to the restriction functor from
$\bG$-actions to $\bH$-actions, from which it follows immediately that
it preserves ergodicity: the induction to $\bG$ of an ergodic
$\bH$-action is again ergodic.

The following familiar concept will allow us to further explicate the
ergodic actions of the compact groups we study.

\begin{definition}\label{def.proj}
  Lets $\bG$ be a compact group, $V$ a finite-dimensional Hilbert
  space, and $U(V)$ its unitary group. A \textit{projective unitary
    representation} of $\bG$ on $V$ is a map $\pi:\bG\to U(V)$ such
  that
\begin{equation}
  \pi(x)\pi(y)=\lambda(x,y)\pi(xy)\quad\forall x,y\in \bG
\end{equation}
where $\lambda:G\times G\to U(V)$ is called the associated multiplier.
\end{definition}

It is easy to see that if $V$ is an {\it irreducible} projective
representation of the compact subgroup $\bH\subseteq \bG$ then $B(V)$
is ergodic over $\bH$ and hence $\ind_{\bH}^{\bG} B(V)$ is an ergodic
$\bG$-action.

\begin{definition}\label{def.rig}
  A compact group $\bG$ is {\it ergodically rigid} if all of its
  ergodic representations are of the form $\ind_{\bH}^{\bG}B(V)$ for
  some closed subgroup $\bH\le \bG$ and some irreducible projective
  $\bH$-representation $V$ of $\bH$.
\end{definition}

Recall \cite[Theorem 20]{w1} (slightly paraphrased):

\begin{theorem}\label{th.typ1}
  $\bG$ is ergodically rigid in the sense of \Cref{def.rig} if and
  only if its only ergodic actions are on type-I von Neumann algebras.
  \qedhere
\end{theorem}
Note also that according to \cite[Theorem, p. 309]{w3} $SU(2)$ is
ergodically rigid. Here, we first need the following simple remark.

\begin{lemma}\label{le.ab-rig}
  Abelian compact groups are ergodically rigid.
\end{lemma}
\begin{proof}
  Let $\bG$ be a compact abelian group acting ergodically on a von
  Neumann algebra $M$. For a character $\chi:\bG\to \bS^1$ we denote
  by $M_\chi$ the spectral subspace of $M$ (i.e. those elements of $M$
  which $\bG$ scales via $\chi$).

  According to \cite[Theorem 1 (a)]{w1} we have
  $\mathrm{dim}(M_\chi)\le 1$. For a non-zero $x\in M_\chi$ we have
  \begin{equation*}
    0\ne x^*x\in M_1 = \bC,
  \end{equation*}
  meaning that $x$ is a scalar multiple of the identity. Those $\chi$
  for which $M_\chi\ne 0$ then form a subgroup
  \begin{equation*}
    \widehat{\bG/\bH}\le \widehat{\bG}
  \end{equation*}
  of the character group of $\bG$ (for some closed subgroup
  $\bH\le \bG$), and we have
  \begin{equation*}
    M\cong \ind_{\bH}^{\bG} N
  \end{equation*}
  for a {\it full-multiplicity} ergodic action of $\bH$ on a von
  Neumann algebra $N$ in the sense of \cite{w2}, i.e. such that for
  each character $\chi\in \widehat{\bH}$ the spectral space $N_\chi$
  has maximal dimension $1$.

  In turn, \cite[Theorem 2]{w2} then shows that the full-multiplicity
  ergodic actions of $\bH$ are precisely $B(V)$ for irreducible
  projective representations $V$.
\end{proof}

We also need the following result on the persistence of ergodic
rigidity under certain extensions.

\begin{proposition}\label{pr.ext-rig}
  Let
  \begin{equation*}
    1\to \bH\to \bG\to \Gamma\to 1
  \end{equation*}
  be an extension of a finite group $\Gamma$ by an ergodically rigid
  compact group $\bH$. Then, $\bG$ is ergodically rigid.
\end{proposition}
\begin{proof}
  According to the already-cited \cite[Theorem 20]{w1}, it suffices to
  prove that for every ergodic action of $\bG$ on a von Neumann
  algebra $M$, the latter is of type I. Furthermore, recall from
  \cite[Corollary 8]{w1} that every ergodic action is induced from an
  ergodic action of a closed subgroup on a factor, so we may as well
  assume that $M$ is a factor.

  Now consider the von Neumann subalgebra $M^{\bH}$ fixed by $\bH$. It
  is acted upon ergodically by $\Gamma$, and hence is
  finite-dimensional by \cite[Theorem 1 (a)]{w1}.

  Let $p$ be a minimal projection of $M^{\bH}$. The factor $pMp$ then
  admits an ergodic action by $\bH$, and hence, by the assumption of
  ergodic rigidity, must be of type I. Since $M$ is a factor with a
  corner $pMp$ of type I, it must itself be of type I. As anticipated
  above, this finishes the proof via \cite[Theorem 20]{w1}.
\end{proof}

We end this section with the following simple consequence of the
general theory recalled above; it will be of use to us in the
classification results to follow.

\begin{lemma}\label{le.upto}
  Let $\bH_i$, $i=1,2$ be two closed subgroups of a compact group
  $\bG$ with respective irreducible projective representations
  $V_i$. Then, the induced representations $M_i=\ind_{\bH_i}^{\bG}B(V_i)$
  are isomorphic if and only if there is an element $g\in \bG$ such
  that
  \begin{itemize}
  \item $g^{-1}\bH_1 g=\bH_2$;
  \item the pullback through the isomorphism 
    \begin{equation*}
      \mathrm{ad}_{g^{-1}}=g^{-1}\bullet g:\bH_1\to \bH_2
    \end{equation*}
    of $B(V_2)$ is isomorphic to the $\bH_1$-module algebra
    $B(\bH_1)$.
  \end{itemize}
\end{lemma}
\begin{proof}
  The sufficiency of the condition is clear: if an element $g$
  satisfying the two conditions exists, then the action of $g$
  implements an isomorphism
  \begin{equation}\label{eq:iso12}
    M_1=\ind_{\bH_1}^{\bG}B(V_1) \cong \ind_{\bH_2}^{\bG}B(V_2)=M_2. 
  \end{equation}
  Conversely, suppose we have an isomorphism \Cref{eq:iso12}. First,
  according to \cite[Theorem 7]{w1}, $L^{\infty}(\bG/\bH_i)$ are the
  centers of the von Neumann algebras $M_i$ respectively, and are
  hence $\bG$-equivariantly isomorphic.

  The algebras $C(\bG/\bH_i)$ can be extracted as the algebras of
  norm-continuous elements with respect to the $\bG$-actions on $M_i$,
  and are hence once more $\bG$-equivariantly isomorphic. This
  translates to a $\bG$-space homeomorphism $\bG/\bH_1\to
  \bG/\bH_2$. If such a homeomorphism sends the class of $1$ in
  $\bG/\bH_1$ to the class of $g\in \bG$ in $\bG/\bH_2$ then the
  isotropy group $\bH_1$ of the former must coincide with the isotropy
  group $g\bH_2 g^{-1}$ of the latter.

  Upon applying $g$, we may now assume that $\bH_i$ coincide (and
  hence drop the subscripts $i$ from $\bH$). The hypothesis is now
  that
  \begin{equation}\label{eq:bdls}
    \ind_{\bH}^{\bG}B(V_1)\cong \ind_{\bH}^{\bG}B(V_2)
  \end{equation}
  via an isomorphism that identifies the centers $L^{\infty}(\bG/\bH)$
  of the two respective sides. The $C^*$-algebras of norm continuity
  on the two sides of \Cref{eq:bdls} are the algebras of continuous
  sections of the bundles over $\bG/\bH$ associated to the actions of
  $\bH$ on $B(V_i)$.

  The desired conclusion that $B(V_i)$ are isomorphic as $\bH$-module
  algebras now follows by evaluating sections of said bundles at the
  class of $1\in \bG$ in $\bG/\bH$.
\end{proof}

\section{Classification results for the compact quantum group $O_{-1}(2)$}\label{se.$O_{-1}(2)$}
In this section we first describe the ergodic actions of $O_{-1}(2)$
and we apply the results of the previous section to obtain the list of
embeddable ergodic actions.

\subsection{Ergodic actions of $O_{-1}(2)$}
In this subsection, we will give the complete list of ergodic
coactions of $O_{-1}(2)$. Let's recall \cite[Theorem 4.3]{BBC}:
 
\begin{theorem}\label{monoi}
  The category of corepresentations of $C(O^{-1}_{n})$ is tensor
  equivalent to the category of representations of $O_{n}$.
\end{theorem}
By \Cref{monoi} the compact quantum groups $O_{-1}(2)$ and $O(2)$ are
monoidally equivalents and by \Cref{ergo_corres} their respective
ergodic actions of $O(2)$ are in bijective correspondence. It thus
suffices to classify the ergodic actions of $O(2)$.

In the sequel we will identify $O(2) \cong T\rtimes C_2$, with
$T = \bS^1$ being the circle group, and with the cyclic group
$C_2 = \{1,\sigma\}$ acting on $T$ by $\sigma(z) = \bar{z}$. As a first observation, we have

\begin{theorem}\label{th.list}
  The compact group $O(2)$ is ergodically rigid in the sense of
  \Cref{def.rig}.
\end{theorem}
\begin{proof}
  Immediate from the expression of $O(2)$ as an extension
  $T\rtimes C_2$ together with \Cref{le.ab-rig,pr.ext-rig}.
\end{proof}

We now describe the ergodic actions more explicitly, via
\Cref{th.list} and the representation theory of the closed subgroups
of $O(2)$. These fall into two classes:
\begin{itemize}
\item the closed subgroups $C_k\le T$, either cyclic of order $k$ or
  equal to $T$ for $k=\infty$;
\item the dihedral groups $D_k=C_k\rtimes C_2$, where again we set
  $D_k=O(2)$ for $k=\infty$.
\end{itemize}

All irreducible representations of $C_k$ give rise through the
procedure described above, by induction, to the same ergodic action
$\alpha^{(k)}$ of $O(2)$ on
$L^{\infty}(O(2)/C_k) = L^{\infty}(T/C_k) \oplus L^{\infty}(T/C_k)$,
namely
\[\alpha_z(f,g) = (f_z,g_z),\quad \alpha_{\sigma}(f,g) = (g,f),\]
where $f_z$ denotes the $z$-translate of $f$.

As for the $D_k$, we have the action $\alpha= \beta^{(k)}_0$ on
$L^{\infty}(O(2)/D_k)$ coming from the characters of $D_k$, as well as
those induced from $D_k$ from the actions of the latter on $M_2(\bC)$
given by
\[\alpha_z\begin{pmatrix} a & b \\ c & d \end{pmatrix}
  = \begin{pmatrix} a & z^lb \\ z^{-l}c & d\end{pmatrix},\quad
  \alpha_{\sigma}\begin{pmatrix} a & b \\ c & d \end{pmatrix}
  = \begin{pmatrix} d & c \\ b & a\end{pmatrix}\] for positive
integers $0<l<k$. We denote these $O(2)$-actions by
$\beta^{(k)}_{l/2}$ respectively (with $k=\infty$ corresponding to the
finite-dimensional ergodic actions of $D_\infty=O(2)$ itself).

All in all, we obtain
\begin{proposition}\label{pr.o2-class}
  The full list of mutually non-equivalent ergodic actions of $O(2)$
  is
  \begin{equation*}
    \left\{\beta^{(k)}_{l/2},\alpha^{(k')}\mid k,k'\in \bZ_{\ge
        0}\cup\{\infty\}, 0\leq l\le \left\lfloor \frac k2\right \rfloor\right\}.
  \end{equation*}
\end{proposition}
\begin{proof}
  The fact that this list contains all (isomorphism classes of)
  ergodic actions follows from \Cref{th.list}, while the claim about
  their being mutually non-isomorphic is a consequence of
  \Cref{le.upto}.
\end{proof}

\subsection{Embeddable ergodic actions of $O_{-1}(2)$} 
\label{subse.o-1-emb}

In this subsection we determine the embeddable ergodic actions on
$O_{-1}(2)$, based on those of $O(2)$ classified above in
\Cref{pr.o2-class}. The plan for achieving this is as follows.

First, note that by definition an embeddable ergodic action is by
definition a comodule $*$-algebra of the CQG algebra $\cA_{-1}$
associated to $O_{-1}(2)$ which embeds into $\cA_{-1}$ as such
(i.e. by an embedding that preserves all of the structure: comodule,
algebra, etc.).
 
Since the twisting equivalence $\lambda\triangleright$ that implements
\Cref{ergo_corres} also implements an equivalence between the
categories of coideal $*$-algebras over $\cA_{-1}$ and the untwisted
version $\cA$ (algebra of representative functions on the classical
group $O(2)$), it will be sufficient to identify the ergodic
$O(2)$-action $\cB$ in the list of \Cref{pr.o2-class} for which
\begin{equation*} \lambda\triangleright \cB \cong \cA_{-1}
\end{equation*} as $\cA_{-1}$ comodule $*$-algebras, and to then also
identify the members of that list that embed into $\cB$.
 
We will see that there is only one candidate for $\cB$ (namely
$\beta^{(2)}_{1/2}$) using the Peter-Weyl theorem to determine the
representation type of the ergodic actions identified in
\Cref{pr.o2-class} (where by representation type we mean the
multiplicities of the various irreducible
$O(2)$-representations). Indeed, this is the substance of the
following result.

\begin{proposition}\label{pr.just-two} The only comodule algebras
among those in \Cref{pr.o2-class} that are isomorphic to $\cA$ as
$O(2)$-representations are $\alpha^{(1)}\cong \cA$ itself and
$\beta^{(2)}_{1/2}$.
\end{proposition}
\begin{proof} The $(\infty)$-superscript $O(2)$-representations are
finite-dimensional, so we can discount them for the purposes of this
proposition.
 
For the other members of the list, we will use the Frobenius
reciprocity formula
\begin{equation}
  \label{eq:frob} \mathrm{hom}_{O(2)}(V,\ind_H^{O(2)}W) \cong
\mathrm{hom}_H(V,W)
\end{equation} for $V\in \Rep_{O(2)}$ and $W\in \Rep_H$ in order to
compute the multiplicities of various irreducible
$O(2)$-representations.

For each $k\ge 1$ we have a $2$-dimensional $O(2)$-representation
$V_k$ whose restriction to $T$, upon identifying the Pontryagin dual
\begin{equation*} 
  \widehat{T} \cong \bZ, 
\end{equation*} 
splits as $k\oplus (-k)$.  
 
Now, for $k\ge 2$, $\alpha^{(k)}$ is induced from the non-trivial
cyclic group $C_k\subset T$. Taking $H=C_k$, $W$ to be trivial, and
$V=V_1$ in \Cref{eq:frob}, the right hand side vanishes and hence so
must the left hand side. This means that $V_1$ is not a summand of
$\alpha^{(k)}$, $k\ge 2$, and hence these list members can also be
dropped as candidates for an isomorphism to $\cA$ as $\cA$-comodules.

Next we look at the representations $\beta_0^{(k)}$ for all $k\ge 1$
induced from the trivial representation of the order-$2k$ dihedral
groups $D_k\subset O(2)$. In these cases, \Cref{eq:frob} with $H=D_k$,
$W$ trivial and $V$ being the non-trivial character of $O(2)$
annihilates the right hand side, and hence the left hand side too. In
conclusion, the non-trivial character of $O(2)$ does not appear in
$\beta_0^{(k)}$; this disqualifies these representations.
 
Finally, we consider $\beta_{\ell /2}^{(k)}$ for $\ell>0$ and
$k\ge 2$. Here, we apply \Cref{eq:frob} with $H=D_k$, $W$ the
representation of $D_k$ on $M_2$ described in the discussion preceding
\Cref{pr.o2-class}, and $V=V_1$. There are now a few possibilities:

{\bf (a)} If $\ell>1$ then the right hand side of \Cref{eq:frob} is
zero, so these cases can be discarded;

{\bf (b)} If $\ell=1$ and $k\ge 3$ then the right hand side of
\Cref{eq:frob} is one-dimensional, because the restriction of $V_1$ to
$D_k$ is irreducible. In conclusion $V_1$ appears in
$\beta_{\ell/2}^{(k)}$ with multiplicity one, but it appears in $\cA$
with multiplicity two (by Peter-Weyl, since it is a two-dimensional
irreducible representation). Once more, these cases do not qualify for
the purposes of the proposition;

{\bf (c)} Finally, $\ell=1$ and $k=2$ is left, in which case one
easily checks that the multiplicities match as expected. Indeed, $D_k$
is then the Klein group $\bZ_2^2$, and its $4$-dimensional
representation $W$ that is induced up to $O(2)$ to produce
$\beta^{(2)}_{1/2}$ breaks up as a sum of all of its characters.

It follows from the previous paragraph that if the irreducible
$O(2)$-representation $V$ is one-dimensional then the right hand side
of \Cref{eq:frob} is also one-dimensional, whereas if $V$ is
two-dimensional then its restriction to $D_2$ breaks up as a sum of
two distinct characters, and hence the right hand side of
\Cref{eq:frob} is two-dimensional.

This finishes the proof of the proposition. 
\end{proof}

\begin{remark}
  In the sequel, we will make repeated and implicit use of the fact
  that in the Frobenius reciprocity formula \Cref{eq:frob}, when $V$
  and $W$ are algebras in the respective categories of
  representations, $\ind^{O(2)}_HW$ is again an algebra in
  $\Rep_{O(2)}$.

  Moreover, \Cref{eq:frob} identifies the subspaces of algebra
  morphisms (i.e. those morphisms that are multiplicative in addition
  to being $O(2)$ and $H$-module maps).
\end{remark}

We can now record the consequence alluded to above.

\begin{corollary}\label{cor.cls}
  The twisting equivalence $\lambda\triangleright$ induces a bijection
  between
  \begin{equation*}
    \{\alpha^{(k)},\ \beta^{(k)}_{l/2}\ |\ k=\infty \text{ or even },\ l=0 \text{ or odd }\}
  \end{equation*}
  from \Cref{pr.o2-class} and the isomorphism classes of embeddable
  ergodic actions of $O_{-1}(2)$.
\end{corollary}
\begin{proof}
  The function algebra of $O_{-1}(2)$ can be obtained from that of
  $O(2)$ by twisting the multiplication both on the right and the
  left, by the cocycle $\lambda$ and its convolution inverse
  $\lambda^{-1}$. Since $\lambda\triangleright$ by definition twists by
  $\lambda$ on the right, the $\cA$-comodule algebra $\cB$ from the
  introductory remarks to \Cref{subse.o-1-emb} is a twist of $\cA$ on
  the left and hence cannot be abelian, and yet must have the same
  representation type as $\cA$ as a right $\cA$-comodule. It must thus
  be $\beta^{(2)}_{1/2}$ by \Cref{pr.just-two}.

  In summary, the desired conclusion will follow once we show that the
  ergodic $O(2)$-actions listed in the statement are precisely those
  that embed into $\beta^{(2)}_{1/2}$.

  Throughout the proof, we denote by $W$ the $D_2$-representation on
  $M_2$ that gives rise to $\beta^{(2)}_{1/2}$ by induction to
  $O(2)$. We examine the representations listed in \Cref{pr.o2-class} systematically.

\vspace{.5cm}

{\bf Type-$\alpha$ actions.}

$\alpha^{(\infty)}$ is two-dimensional. Its restriction to $H=D_2$
embeds into $W$ as the diagonal subalgebra of the realization of $W$
as $2\times 2$ matrices, and hence $\alpha^{(\infty)}$ embeds into
$\beta^{(2)}_{1/2}$ by Frobenius reciprocity \Cref{eq:frob}.

As for $\alpha^{(k)}$ for positive integers $k$, consider first the
case when $k$ is odd. If we had an embedding
\begin{equation*}
  \alpha^{(k)} \subseteq \beta^{(2)}_{1/2},
\end{equation*}
then the Frobenius adjunction \Cref{eq:frob} would turn it into a map 
\begin{equation}\label{eq:alg-map}
  \res^{O(2)}_H\alpha^{(k)}\to W
\end{equation}
of algebras in $\Rep_{D_2}$. The condition that $k$ be odd then
ensures that this map is surjective, since in that case all four
characters of $D_2$ admit unitary eigenvectors in the restriction of
$\alpha^{(k)}$. Since however the left hand side of \Cref{eq:alg-map}
is commutative while the right hand side is not, we obtain a
contradiction.

For even $k$ on the other hand, we can embed $\alpha^{(k)}$ into
$\beta^{(2)}_{1/2}$ by inducing in stages. First, embed
\begin{equation*}
  \res^{D_k}_{D_2}\ind_{C_k}^{D_k}\bC \subseteq W
\end{equation*}
as the diagonal subalgebra of the $2\times 2$ matrix realization of
$W$. Frobenius reciprocity then translates this into an embedding
\begin{equation*}
  \ind_{C_k}^{D_k}\bC \subseteq \ind_{D_2}^{D_k} W.
\end{equation*}
Finally, induce this map further to $O(2)$.

\vspace{.5cm}

{\bf Type-$\beta$ actions, $l=0$.}

$\beta^{(\infty)}_0$ is simply the trivial representation and hence is
embeddable into $\beta^{(2)}_{1/2}$. We note also that $\beta^{(k)}_0$
for odd $k$ can be eliminated in exactly the same way we did
$\alpha^{(k)}$ above.

For even $k$ $\beta^{(k)}_0$ is again embeddable into $\beta^{(2)}_{1/2}$ by the case of even $\alpha^{(k)}$, since we have 
\begin{equation*}
  \beta^{(k)}_0 \subset \alpha^{(k)}. 
\end{equation*}

\vspace{.5cm}

{\bf Type-$\beta$ actions, $l>0$.}

Consider the case of $\beta^{(k)}_{l/2}$ (including $k=\infty$) for
even positive $l$. Here we have an embedding
\begin{equation}\label{eq:infk}
  \beta^{(\infty)}_{l/2} \subseteq \beta^{(k)}_{l/2} 
\end{equation}
of algebras in $\Rep_{O(2)}$, and hence an embedding of the right hand
side into $\beta^{(2)}_{1/2}$ would imply the existence of a morphism
of the left hand side into $W$ in the category $\Rep_{D_2}$. This is
impossible: both the left hand side of \Cref{eq:infk} and $W$ are
$2\times 2$ matrix algebras and hence the morphism would have to be
one-to-one, but the evenness of $l$ ensures that when restricted to
$D_2$ the left hand side of \Cref{eq:infk} has a two-dimensional space
of invariants.

When $k$ is positive and odd, then for every $l$ we have an even $l'$ such that 
\begin{equation*}
  l'\equiv l~ (\mathrm{mod}~k).
\end{equation*}
We have an embedding  
\begin{equation*}
  \beta^{(\infty)}_{l'/2} \subseteq \beta^{(k)}_{l/2} 
\end{equation*}
of algebras in $\Rep_{O(2)}$ and we can repeat the argument above to
conclude that $\beta^{(k)}_{l/2}$ is not embeddable into
$\beta^{(2)}_{1/2}$.

For even $k$ (including by abuse the case $k=\infty$ with $D_k=O(2)$)
and positive odd $l$ the restriction of $\beta^{(\infty)}_{l/2}$ to
$D_k$ embeds into
\begin{equation*}
  \ind_{D_2}^{D_k} W,
\end{equation*}
and hence $\beta^{(k)}_{l/2}$ is embeddable into $\beta^{(2)}_{1/2}$,
as desired.

This concludes the last case and the proof of the result.
\end{proof}


\subsection{Quotients by quantum subgroups}
\label{subse.quot}

In this section we identify those embeddable ergodic actions that
arise as function algebras of quotients by quantum subgroups of
$O_{-1}(2)$. 

We denote by $\cH=\cA_{-1}$ the Hopf algebra underlying
$O_{-1}(2)$. The Hopf $*$-algebra quotients of $\cH$ (i.e. the
function algebras of the quantum subgroups of $)_{-1}(2)$) are
classified in \cite[Theorem 7.1]{bb-4pts}. We briefly recall that
classification here. The non-trivial quotients are as follows. 
\begin{itemize}
\item For each $n\in \bZ_{>0}\cup\{\infty\}$ a quotient isomorphic to
  the group group algebra $\bC D_n$ of the dihedral group of order
  $2n$ (including $n=\infty$);
\item Two families of Hopf algebras $A(n,e)$, $e=\pm 1$,
  $n\in \bZ_{>0}$ of respective orders $4n$.
\end{itemize}

For each quotient Hopf $*$-algebra $\pi:\cH\to \cL$ we have an
associated right coideal $*$-subalgebra
\begin{equation}\label{eq:quot}
  \cA=\cA_\pi = \{x\in \cH\ |\ (\pi\otimes\mathrm{id})\Delta(x) = 1\otimes x\in \cL\otimes \cH\}.
\end{equation}
Our first remark identifies those embeddable actions that can be
realized as such coideal subalgebras for the quotients $\cL=A(n,e)$
from the above classification.

\begin{proposition}\label{pr.ane}
  Let $n\in \bZ_{>0}$. The coideal subalgebras corresponding to
  $\cH\to A(n,e)$, $e=\pm 1$ are isomorphic to the ergodic action
  $\beta^{(2n)}_0$ from \Cref{cor.cls}.
\end{proposition}
\begin{proof}
  It is easy to see from the proofs of \Cref{pr.just-two,cor.cls} that
  $\beta^{(2n)}_{0}$ are the only embeddable coactions among those in
  \Cref{cor.cls} that do not contain the non-trivial one-dimensional
  comodule of $\cH$.

  On the other hand, it follows from \cite[Lemma 7.3 and Theorem
  7.1]{bb-4pts} that the quotients $\pi_{n,e}:\cH\to A(n,e)$ are those
  for which the non-trivial grouplike $d\in \cH$ satisfies
  $\pi(d)\ne 1$; by the previous paragraph, it follows that the
  comodule algebras corresponding to the actions $\beta^{(2n)}_0$ are
  indeed among $\cA_{\pi_{n,e}}$ defined as in \Cref{eq:quot}.

  Now consider the simple two-dimensional $\cH$-comodules $V_k$,
  $k\in \bZ_{>0}$ corresponding to the simple $O(2)$-representations
  denoted by the same symbols in the proof of \Cref{pr.just-two}. It
  follows from \cite[Lemma 7.4]{bb-4pts} (and its proof) that when
  regarded as a comodule over $A(n,e)$, $V_k$ contains the trivial
  representation precisely when $2n$ divides $k$. This, then, is the
  sufficient and necessary condition that ensures that $V_k$ appears
  as a subcomodule of $\cA_{\pi_{n,e}}$.

  The conclusion now follows from the observation that, by Frobenius
  reciprocity, $V_k$ is similarly embeddable into $\beta^{(2n)}_0$ as
  a comodule if and only if $2n|k$.
\end{proof}

It now remains to identify those embeddable ergodic actions that
correspond to the quantum subgroups $\cH\to \bC D_n$ for
$n\in \bZ_{>0}\cup \{\infty\}$. According to \Cref{cor.cls,pr.ane},
these will be among the $\alpha^(k)$ and $\beta^{(k)}_{l/2}$ for even
$k$ (including $k=\infty$) and odd $l$.

\begin{proposition}\label{pr.dih}
  Let $n\in \bZ_{>0}\cup\{\infty\}$. The $\cH$-comodule algebra
  $\alpha^{(2n)}$ is isomorphic to the right coideal subalgebra
  $\cA_{\pi_n}$ of $\cH$ associated to the Hopf quotient
  $\pi_n:\cH\to \bC D_n$.
\end{proposition}
\begin{proof}
  As in the proof of \Cref{pr.ane} above, denote by $V_k$ the simple
  two-dimensional $\cH$-comodules for $k\in \bZ_{>0}$. Similarly, let
  $C_k\subset \cH$ be the corresponding $2\times 2$ matrix coalgebra.

  The explicit description of the matrix coalgebras $C_k$ from
  \cite[discussion preceding Proposition 7.1]{bb-4pts} shows that
  $C_k$ is contained in $\cA_{\pi_n}$ when $2n|k$, and intersects
  $\cA_{\pi_n}$ trivially otherwise.

  The statement is now a consequence of the fact that similarly, the
  multiplicity of $V_k$ in $\alpha^{(2n)}$ is two when $2n|k$ and zero
  otherwise.
\end{proof}

\subsection{Generalized quantum subgroups}
\label{subse.gen-quot}

As seen in \Cref{subse.quot} above, the quantum subgroups of
$O_{-1}(2)$ do not account for all embeddable ergodic actions of the
latter quantum group. We will see here that nevertheless, these
ergodic actions can be recovered through what might be deemed
``subquotient'' quantum groups of $O_{-1}(2)$. To make sense of this,
we need to recall some material from \cite{fr-sk-id}.

First, consider an arbitrary CQG algebra $\cH$. \cite[Theorem 1]{ss17}
establishes a one-to-one correspondence between certain coideal
subalgebras of $\cH$ (which are morally the embeddable actions of the
underlying quantum group of $\cH$) and the {\it idempotent states} on
the latter, i.e. those states $\phi$ satisfying $\phi*\phi=\phi$ for
the convolution product.

An idempotent state is a generalization of a quantum subgroup, since
given such a quantum subgroup $\pi:\cH\to \cB$ the composition
$h_{\cB}\circ \pi$ is idempotent. For this reason, the coideal
subalgebra of $\cH$ defined by
\begin{equation*}
  \mathrm{Im}(\phi\otimes\mathrm{id})\circ\Delta
\end{equation*}
for an idempotent state $\phi$ can be regarded as a natural
generalization of a quotient by a quantum subgroup.

Now suppose $\cH=\bC\Gamma$ is the group algebra of a discrete group
(i.e. the Hopf algebra underlying an {\it abelian} compact quantum
group). As seen in \cite[Theorem 6.2]{fr-sk-id} (for finite groups but
the discussion generalizes), the idempotent states on $\cH$ are simply
the characteristic functions of subgroups of $\Gamma$.

In general, for an arbitrary CQG algebra $\cH$ with a quotient
$\cH\to \bC\Gamma$, the characteristic function on a subgroup of
$\Gamma$ is an idempotent state on $\cH$ and hence corresponds to some
coidalgebra of $\cH$. With this in mind, we introduce the following
term to aid the streamlining of the presentation.

\begin{definition}\label{def.tame}
  Let $\cH$ be a CQG algebra, $\pi:\cH\to \bC\Gamma$ a quotient group
  algebra, and $\Omega\subset\Gamma$ a discrete subgroup. We denote
  \begin{equation*}
    \cA_{\pi,\Omega} = \mathrm{Im}(\phi\otimes\mathrm{id})\circ\Delta,
  \end{equation*}
  where $\phi:\cH\to \bC$ is the characteristic function on
  $\Omega\subset\Gamma$ composed with $\pi$.

  A {\it tame} embeddable ergodic action of the quantum group attached
  to $\cH$ is one that is isomorphic to the coidalgebra
  $\cA_{\pi,\Omega}$ for some $\pi:\cA\to \bC\Gamma$ and some subgroup
  $\Omega\subseteq \Gamma$.
\end{definition}

This notion allows us to draw the conclusion announced above.

\begin{proposition}\label{pr.tame}
  The ergodic actions $\beta^{(k)}_{l/2}$, $l\ne 0$ of $O_{-1}(2)$
  listed in \Cref{cor.cls} are tame in the sense of \Cref{def.tame}.
\end{proposition}
\begin{proof}
  Specifically, we will show that all of these are isomorphic to
  coidalgebras $\cA_{\pi,\Omega}$ where $\pi:\cH\to \bC D_\infty$ is
  the surjection onto the group algebra of the infinite dihedral group
  from \cite[Theorem 7.1]{bb-4pts} and $\Omega\subseteq D_\infty$ are
  various subgroups.

  The discussion in \cite[Section 7]{bb-4pts} introduces the matrix
  counits $v_{ij}$ for the comodule $V_1$ of $\cH$, and the quotient
  $\cH\to \bC D_\infty$ sends $v_{ii}$, $i=1,2$ to the two involutions
  $\sigma_i$ generating $D_\infty$ and annihilates $v_{ij}$, $i\ne j$.

  Furthermore, the matrix subcoalgebra $C_k\subset \cH$ associated to
  the simple two-dimensional $\cH$-comodule $V_k$, $k\in \bZ_{>0}$ is
  \begin{equation*}
    \begin{pmatrix}
      (v_{11}v_{22})^mv_{11}^\varepsilon & (v_{12}v_{21})^mv_{12}^\varepsilon
    \\(v_{21}v_{12})^mv_{21}^\varepsilon & (v_{22}v_{11})^mv_{22}^\varepsilon
    \end{pmatrix}
  \end{equation*}
  where $\varepsilon\in \{0,1\}$ and $k=2m+\varepsilon$.

  Now let $k$ be even or $\infty$ and $\ell$ odd, parametrizing the
  actions $\beta^{(k)}_{l/2}$ from \Cref{cor.cls}.

  By simply counting multiplicities of the various $V_t$, the explicit
  description of the matrix coalgebras $C_k$ now makes it an easy
  check that $\beta^{(k)}_{l/2}$ is isomorphic as an $\cH$-comodule to
  $\cA_{\pi,Omega}$, where the subgroup $\Omega$ of $D_\infty$ is the
  semidirect product of the subgroup of index $k$ in
  $\bZ\subset D_\infty$ by the order-two group generated by
  $(g_1g_2)^{\frac {l-1}2}g_1$.

  The conclusion follows from this, since the comodule algebras in
  \Cref{cor.cls} are mutually non-isomorphic as comodules.
\end{proof}

In conjunction with \Cref{pr.ane,pr.dih}, this result accounts for all
of the ergodic actions of $O_{-1}(2)$ as classified in \Cref{cor.cls}.


\section{Counterexamples: dihedral groups}
\label{se.cntrex}

Recall \Cref{qu.mon-erg}, on whether or not cocycle-twisting in some
sense preserves isomorphism classes of embeddable ergodic actions. One
possible precise interpretation would be as follows (in the context of
compact quantum groups $\bG_i$ obtained via cocycle deformation for a
cocycle $\lambda$).

\begin{question}\label{qu.mon-erg-bis}
  Does 
  \begin{equation*}
    \lambda\triangleright:\cE rg(\bG_1) \to \cE rg(\bG_2)
  \end{equation*}
  restrict to an equivalence between subcategories of {\it embeddable}
  ergodic coactions?
\end{question}

We already know that the answer to this version of the question is
negative, by examining the mutual twists $O(2)$ and $O_{-1}(2)$ we
have been studying:

\begin{corollary}\label{cor.nogo1}
  Let $\lambda$ be a cocycle which twists $O(2)$ into
  $O_{-1}(2)$. Then, the answer to \Cref{qu.mon-erg-bis} is negative
  for $\bG_1=O(2)$ and $\bG_2=O_{-1}(2)$.
\end{corollary}
\begin{proof}
  This is an immediate consequence of \Cref{cor.cls}. 
\end{proof}

The question remains however of whether one can implement a more
sophisticated equivalence between the embeddable ergodic actions of
two mutual twists. To rule this out, we will observe below that there
are examples of mutually cocycle-twisted {\it finite} quantum groups
with different numbers of isomorphism classes of embeddable ergodic
actions.

The groups in question will be discrete versions of $O(2)$, i.e. the
dihedral groups $D_K$ (for even $K$). The contents of this section can
thus be regarded as a ``discretization'' of those of
\Cref{se.$O_{-1}(2)$}. We will mostly omit proofs, as they are almost
verbatim recapitulations of those in the preceding section.

Fix an even positive integer $K$ (though evenness will only be of
relevance to parts of the discussion below).

Then, for the order-$2K$ dihedral group $D_K$, we preserve the
notation $\alpha^{(k)}$ and $\beta^{(k)}_{\frac l2}$ for
representations induced from subgroups $C_k$ and $D_k$ of $K$. Note
that whenever we employ this notation, the condition $k|K$ is
implicit.

The classification of ergodic actions is perfectly analogous to that
in \Cref{pr.o2-class} with a parallel proof, via ergodic rigidity and
an appeal to \Cref{pr.ext-rig,le.ab-rig,le.upto}).

\begin{proposition}\label{pr.d-class}
  The full list of mutually non-equivalent ergodic actions of $D_K$ is
  \begin{equation*}
    \{\beta^{(k)}_{l/2},\alpha^{(k')}\}
  \end{equation*}
  for $k,k' | K$ and $0\leq l\le \left\lfloor \frac k2\right\rfloor$. \qedhere
\end{proposition}

The usual cocycle used to twist $O(2)$ into $O_{-1}(2)$ descends to a
cocycle on the function algebra of $D_K\subset O(2)$, so it can be
used to twist the latter into $(D_K)_{-1}$. We preserve the notation
$\lambda$ for the cocycle.

Pursuing the same strategy as for $O(2)$, we can now classify the
embeddable ergodic actions of $(D_K)_{-1}$ as an analogue of
\Cref{cor.cls}.

\begin{proposition}\label{pr.d-emb}
  The twisting equivalence $\lambda\triangleright$ induces a bijection
  between
  \begin{equation*}
    \{\alpha^{(k)},\ \beta^{(k)}_{l/2}\}
  \end{equation*}
  from \Cref{pr.d-class} for even $k|K$ and $l=0$ or odd and the
  isomorphism classes of embeddable ergodic actions of $(D_K)_{-1}$.
  \qedhere
\end{proposition}

Finally, an immediate consequence of this of relevance to
\Cref{qu.mon-erg} is

\begin{corollary}\label{cor.nogo2}
  For infinitely many $K$ the sets of isomorphism classes of
  embeddable ergodic actions for $D_K$ and $(D_K)_{-1}$ have different
  cardinalities.
\end{corollary}
\begin{proof}
  This is a simple numerical estimate based on the classification of
  embeddable ergodic actions of $(D_K)_{-1}$ from
  \Cref{pr.d-emb}. That result shows that the number of isomorphism
  classes for $(D_K)_{-1}$ grows quadratically with the number of
  divisors of $K$. On the other hand, for $D_K$, the ergodic actions
  that are embeddable are the $\alpha$s and those $\beta$s in
  \Cref{pr.d-class} with $l=0$. In conclusion, the number of such
  isomorphism classes grows linearly with the number of divisors of
  $K$.
\end{proof}

\nocite{crespo}           \nocite{DC}                        \nocite{BC}  
 

\bibliographystyle{plain}


\Addresses

\end{document}